\documentclass[10pt, conference, letterpaper, twocolumn]{ieeeconf}
\IEEEoverridecommandlockouts
\usepackage{cite}
\usepackage{amsmath,amssymb, amsfonts}
\usepackage{graphicx}
\usepackage{hyperref}
\usepackage{textcomp}
\let\labelindent\relax
\usepackage{enumitem}
\usepackage{xcolor}
\def\BibTeX{{\rm B\kern-.05em{\sc i\kern-.025em b}\kern-.08em
    T\kern-.1667em\lower.7ex\hbox{E}\kern-.125emX}}

\usepackage{amsthm}

\usepackage{thmtools} 
\usepackage{thm-restate}
\usepackage{orcidlink}

\newtheorem{theorem}{Theorem}

\newtheorem*{remark*}{Remark}
\newtheorem{lemma}{Lemma}
\newtheorem*{lemma*}{Lemma}
\newtheorem{assumption}{Assumption}

\DeclareMathOperator*{\argmin}{arg\,min}
\DeclareMathOperator*{\Diag}{Diag}

\DeclareMathOperator*{\interior}{int}
\DeclareMathOperator*{\Sym}{Sym}
\DeclareMathOperator*{\dist}{dist}

\newcommand{\X}{\mathcal{X}}
\newcommand{\R}{\mathbb{R}}
\newcommand{\D}{\operatorname{D}\!}

\renewcommand{\P}{\mathcal{P}}
\newcommand{\I}{\mathbb{I}}
\newcommand{\LDPhi}{L_{2}}
\newcommand{\LDG}{L_{1}}
\newcommand{\Lg}{L_{3}}

\author{
Georgios Vasileiou\orcidlink{0009-0002-3679-0510}$^{\dagger}$,
Lantian Zhang\orcidlink{0000-0002-1814-5596}$^{\dagger}$,
and
Silun Zhang\orcidlink{0000-0003-3772-761X}$^{\dagger}$%
\thanks{This work has been partially supported by the Wallenberg AI, Autonomous Systems and Software Program (WASP), funded by the Knut and Alice Wallenberg Foundation.}%
\thanks{$^{\dagger}$ Georgios Vasileiou,
Lantian Zhang
and Silun Zhang are with the Department of Mathematics at the KTH Royal Institute of Technology, Stockholm, 10044, Sweden. {\tt\small\{geovas, 
lantian,
silunz\}@kth.se}}%
}%

\begin{document}
\title{\LARGE \bf
Incentive Design without Hypergradients: A Social-Gradient Method}
\maketitle
\thispagestyle{plain}
\pagestyle{plain}
\bstctlcite{IEEE:BSTurloff}

\begin{abstract}

In incentive design problems, a system planner seeks to steer self-interested agents toward a socially optimal Nash equilibrium by issuing incentives in the presence of information asymmetry, that is, uncertainty about the agents' cost functions.
A common approach formulates the problem as a Mathematical Program with Equilibrium Constraints and optimizes incentives using hypergradients—the total derivatives of the planner's objective with respect to incentives.
However, computing the hypergradient typically requires full or partial knowledge of equilibrium sensitivities to incentives,
which is generally unavailable under information asymmetry. In this paper, we propose a hypergradient-free incentive law, called the \emph{social-gradient flow}, for incentive design when the social cost depends on the agents' joint action.
We prove that the social cost gradient is a descent direction for the planner's objective,
irrespective of the agent cost landscape.
In the idealized setting where equilibrium responses are observable, the social-gradient flow converges to 
the unique socially optimal incentive.
When equilibria are not directly observable, the social-gradient flow emerges as the slow-timescale limit of a two-timescale interaction, in which agents' strategies evolve on a faster timescale. 
We establish that the joint strategy-incentive dynamics converge to the social optimum for any agent learning rule that asymptotically tracks
the equilibrium. Finally, we validate theoretical results via numerical experiments.
\end{abstract}

\section{Introduction}

Incentive design addresses the problem of a system planner who seeks to steer the collective behavior of self-interested agents toward a socially desirable outcome by publicly committing to a reward (or penalty) rule.
This rule can be interpreted as a feedback law between the planner and agents~\cite{hoControltheoreticViewIncentives1980}, aligning agents' preferences with a target equilibrium~\cite{ratliffPerspectiveIncentiveDesign2019}.
This problem results in a hierarchical decision-making structure and is encountered in many
engineered systems such as demand response pricing in energy markets~\cite{liDistributedOnlinePricing, fochesatoStackelbergGameIncentivebased2022, liSociallyOptimalEnergy2024},
and congestion-aware tolling~\cite{povedaDistributedAdaptivePricing,
barreraDynamicIncentives}.

A central difficulty in incentive design is the \emph{information asymmetry} among participants, in which the leader lacks information on agents’ cost functions. For example, energy providers design time-of-use tariffs without full knowledge of consumers’ preferences~\cite{freir2022,sastry2015theftDetection}, and tolling authorities cannot directly observe commuters' route-choice criteria~\cite{kenneth2005MotoristPreferences}.
Such incentive problems can be formulated as reverse Stackelberg games with incomplete information on agents'  costs~\cite{hoInformationStructure1981, ratliffPerspectiveIncentiveDesign2019, velichetiHarnessingInformationIncentive2025}.
When agents' costs are known, a common approach casts incentive problems as Mathematical Programs with Equilibrium Constraints (MPECs)~\cite{luoMathematicalProgramsEquilibrium1996}, 
in which the planner minimizes a social cost subject to the constraint that agents respond with a Nash equilibrium (NE). The MPECs are usually solved using \emph{hypergradient descent} (see, e.g., \cite{liEndtoendLearningIntervention2020, kimMPECMethodsBilevel2020, hauswirthTimescaleSeparationAutonomous2021a}), which updates the incentives by accounting for how changes in the incentive parameters affect the planner’s objective.

Gradient-based approaches for incentive design have been studied in continuous and discrete time.
In continuous-time, recent work has utilized full- or partial-information hypergradient descent to solve Stackelberg games
under different low-level costs and dynamics (see, e.g., \cite{mojica-navaStackelbergPopulationLearning2022,alpcanControlTheoreticApproach2009}). 
In discrete time, timescale separation is leveraged in~\cite{liuInducingEquilibriaSimulatneous} to asymptotically estimate the hypergradient from local cost information.
The work~\cite{maheshwariAdaptiveIncentiveDesign2024} does not rely on hypergradients, yet requires local utility information to compute agent externalities, defined as the marginal difference between the social objective and the agents' costs.
Alternatives to gradient-based methods are given in inverse game theory~\cite{bertsimasDatadrivenEstimationEquilibrium2015a}, reinforcement learning~\cite{ngAlgorithmsInverseReinforcementLearning}, and adaptive incentive design~\cite{ratliffAdaptiveIncentiveDesign2021}.
Recent work has also investigated no-regret incentive design for nonlinear games~\cite{vasileiou2026noregretframeworkadaptiveincentive, vasileiou2026adaptive} and dynamic games~\cite{zhang_adaptive_lq_games}, which jointly address preference learning and equilibrium regulation in a single closed loop. Drawing on advances in nonlinear adaptive control~\cite{zhang2026stochastic,vanweerelt2026self,zhang2025online}, these methods balance exploration for preference identification with incentive adaptation for social coordination, achieving vanishing average regret.

This paper introduces the \textit{social-gradient flow}, a hypergradient-free incentive law for problems in which the social cost depends on agents’ joint action. Unlike the aforementioned methods, the approach uses only the social-cost gradient, without estimating private costs or differentiating through the equilibrium response, and does not require direct equilibrium observation.
Instead, it uses agents' transient trajectories under the assumption that they asymptotically converge to Nash equilibria, enabling the planner to update incentives concurrently with agents' learning.
First, we establish that the social-gradient flow is a descent direction for the planner’s objective, irrespective of the unknown agent cost landscape, and
prove global convergence to the unique socially optimal incentive from any initial condition. 
Second, 
when only transient trajectories are observed,
the social-gradient flow is the slow-timescale limit of a two-timescale interaction with the agents. Under the assumption that the agents’ equilibrium-seeking dynamics converge asymptotically, we prove the joint strategy-incentive dynamics converge to the social optimum.

\textbf{Notation.}
For $f:\R^n \to \R$, denote by $\nabla f(x)$ its gradient at $x$. We say that $f$ is $\mu$-strongly convex on $\X\subset \R^n$ if $
f(y)\ge f(x)+\langle \nabla f(x),\, y-x\rangle+\frac{\mu}{2}\|y-x\|^2_2,$
for all $x,y\in\X$.
For $F:\R^n \to \R^m$, denote by $\D F(x) \in \R^{m\times n}$ its Jacobian at $x$.  We say that either $F$ or $\D F$ are $L$-Lipschitz on $\X$ if $\|F(x)-F(y)\|_2 \leq L \|x-y\|_2$ or $\|\D F(x)- \D F(y)\|_2 \leq L \|x-y\|_2$, respectively, for all $x,y\in\X$. 
For a matrix $A$, define $\Sym(A) = \frac{1}{2}(A + A^\top)$ as the symmetrization of $A$. 
We write $f(t) = O(g(t))$ if there exist $c,t_0> 0$ such that $|f(t)|\leq c|g(t)|$ for all $t\geq t_0$.
Write $f(t) = o(g(t))$ if $ \lim_{t\to \infty}f(t) /g(t) = 0$.
$C^k$ is the class of $k$-order continuously differentiable functions.
\section{Problem Formulation}
\label{sec:problem_formulation}

We consider a game among $n$ noncooperative agents indexed by $\mathcal{N} = \{1, \ldots, n\}$ and influenced by a single, external system planner. Each agent $i\in \mathcal{N}$ seeks to minimize its own cost by choosing a strategy $x_i \in \mathcal X_i$, with $\mathcal{X}_i\subset \mathbb R$ a compact, convex set.\footnote{We present the scalar case for notational simplicity; the results extend to higher-dimensional strategy spaces.} Let $\X = \prod_{i \in \mathcal{N}}\X_i$ be the joint strategy space.
Agent $i$ is equipped with a private nominal cost $\ell_i:\X\to\R$, where $\ell_i \in C^2$,
and also receives an incentive $\gamma_i:\X \to \R$ issued by the system planner. 
Affine incentives are sufficient to induce desirable behavior by modifying agents' marginal costs~\cite{basarAffineIncentiveSchemes1984, maheshwariAdaptiveIncentiveDesign2024}, and are pointwise optimal among convex maps~\cite[Lemma~1]{vasileiou2026noregretframeworkadaptiveincentive}.
The cost of agent $i$ is
 \begin{equation}\label{eq:indiv_cost}
 c_i^\gamma(x) = \ell_i(x) + \gamma_i (x) =\ell_i(x) + p_ix_i, \quad x\in\mathcal \X, 
\end{equation}
where $ p_i \in \mathbb R$ is the marginal incentive parameter.
We refer to the vector $p = [p_1, \, \ldots, \, p_n]^\top \in \R^n$ as the incentive issued to agents.
Define the pseudo-gradient of the incentivized and nominal games, respectively, as
\[G_\gamma(x) = \begin{bmatrix}
  \frac{\partial}{\partial x_1}c_1^\gamma \\
  \vdots \\
 \frac{\partial}{\partial x_n}c_n^\gamma \\
\end{bmatrix}, \text{ and }
G_0 (x) = \begin{bmatrix}
  \frac{\partial}{\partial x_1}\ell_1 \\
  \vdots \\
 \frac{\partial}{\partial x_n}\ell_n \\
\end{bmatrix}.  \]
Due to~\eqref{eq:indiv_cost}, $p$ shifts the pseudo-gradient according to $G_\gamma(x) = G_0(x) + p$.
Using the equivalence to variational inequality problems~\cite{facchineiFiniteDimensionalVariationalInequalities2004}, 
we define an \emph{incentivized Nash equilibrium} (INE) as any joint action $x^*(p)\in\X$ satisfying
\begin{equation}\label{eq:NE_condition}
G_\gamma(x^*(p))^\top (x-x^*(p))\ge 0,\qquad \forall x\in\X.
\end{equation}

\begin{assumption}
\label{ass:strong_monotonicity}
    There exist positive scalars $\LDG>0$ and $m>0$ such that $\D G_0$ is $\LDG$-Lipschitz in $\X$, and
    such that 
    \begin{equation}
    \begin{aligned}
        & \Sym(\D G_0(x)) \succeq \frac{m}{2} \I, \, \forall x \in \X. \\
    \end{aligned} 
    \end{equation}
\end{assumption}

Assumption~\ref{ass:strong_monotonicity}
implies the existence and uniqueness of the INE for all $p\in \R^n$~\cite[Theorem~2.3.3]{facchineiFiniteDimensionalVariationalInequalities2004}.\footnote{
Assumption~\ref{ass:strong_monotonicity} can be relaxed to Rosen's \emph{diagonal strict convexity (concavity)}  condition~\cite{rosenExistenceUniquenessEquilibrium1965a},
though doing so would 
require an additional diagonal matrix known in the proposed incentive-seeking dynamics.}
We call any map $p \mapsto x^*(p)$ that satisfies~\eqref{eq:NE_condition} a \emph{response map} for the incentivized game.
Define the set $\P = -G_0(\interior \X) \subset \R^n$,
where $\interior \X$ is the interior of $\X$.

\begin{lemma}
  \label{lemma:NE_existence}
  Suppose Assumption \ref{ass:strong_monotonicity} holds. 
  On the open set $\P$, the response map $x^*(\cdot)$
  is unique, and $x^* : \P \to \interior \X$ is a $C^1$ diffeomorphism, i.e., a $C^1$ bijection, with a $C^1$ inverse.
\end{lemma}

The proof of Lemma~\ref{lemma:NE_existence} is given in Appendix~\ref{app:proofs}.
$\P$ is bounded and simply connected, since $x^*: \P \to \interior \X$ is a diffeomorphism and $\interior \X$
is an open hyperrectangle in $\mathbb R^n$. However, $\P$ need not be convex.
The following lemma presents properties of the map $x^*(\cdot)$.
\begin{lemma}
\label{lemma:NE_properties}
  Under Assumption \ref{ass:strong_monotonicity}, 
   it holds that, for any $p \in \P$,
  \begin{enumerate}
    \item[(i)] the response $x^*(p)$ satisfies \(G_\gamma(x^*(p))= 0, \) and the response map $x^*$ is $ \frac{2}{m}$-Lipschitz;
    \item[(ii)] the response map Jacobian $\D x^*(p)$ is non-singular and satisfies $\Sym(\D x^*(p)) \prec 0$; 
    \item[(iii)]  there exist positive constants $\sigma_1$ and $\sigma_2$ such that
    \[\sigma_1 \|w\|_2 \leq \|\D x^*(p) w \|_2 \leq \sigma_2 \|w\|_2, \, \forall w \in \R^n,\]
    where $\sigma_1 \!=\! \left( \max_{x \in \X} \!\|\D G_0(x) \|_2 \right)^{-1} $ and $\sigma_2 = \frac{2}{m}$; and
  \item[(iv)] the Jacobian $\D x^*(p)$ is $\frac{8 \LDG}{m^3}$-Lipschitz in $\P$.
  \end{enumerate}
\end{lemma}

The proof of Lemma~\ref{lemma:NE_properties} is given in Appendix~\ref{app:proofs}.
The system planner's objective is to steer the agents toward the minimizer
of $\Phi:\X \to \R$ by selecting an appropriate incentive $p$.
Hence, the planner solves
\begin{equation}
  \label{eq:planner_problem}
    \min_{p \in \R^n} \Phi(x^*)  \text{ s.t. } G_\gamma(x^*)^\top \left(x - x^*\right) \geq 0, \, \forall x \in \X.
\end{equation}

\begin{assumption}
\label{ass:social_cost}
There exist positive scalars $\mu_\Phi>0$ and $\LDPhi>0$ such that
the social cost $\Phi:\X \to \R$ is $\mu_\Phi$-strongly convex and has a $\LDPhi$-Lipschitz gradient in $\X$.
Further, the unique social optimum $x^\dagger = \argmin_{x \in \X}\Phi(x)$ is in $\interior \X$.
\end{assumption}

Let $p^\dagger = -G_0(x^\dagger) \in \P$ denote the unique incentive satisfying $x^*(p^\dagger) = x^\dagger$, i.e., the incentive that aligns the  INE with the social optimum.
When restricted to $\P$, 
problem~\eqref{eq:planner_problem} reduces to an implicit minimization of $\Phi_*(p) = \Phi(x^*(p))$.
Its (hyper)gradient with respect to $p$ is
\begin{equation*}
  \nabla\Phi_*(p) = \D x^*(p)^\top \nabla \Phi(x)\big|_{x = x^*(p)}.
  \label{eq:hypergradient}
\end{equation*}

The key challenge in solving~\eqref{eq:planner_problem} is that 
both the domain $\P$ and the response map $x^*$ depend on the private costs of the agents. Moreover, 
the hypergradient $\nabla \Phi_*(p)$ is not computable due to its dependence on the INE sensitivity $\D x^*(p)$.
In the sections that follow, we design an incentive adaptation law that asymptotically solves the planner's minimization without requiring agents' private information.

\section{Social-Gradient Incentive Adaptation}
\label{sec:gradient_flow}
This section introduces our proposed \textit{social-gradient flow}, a continuous-time incentive-seeking dynamical system that 
asymptotically converges to the socially-optimal $p^\dagger$, under a direct equilibrium observation model~\cite{ratliffAdaptiveIncentiveDesign2021, liSociallyOptimalEnergy2024} where $x^*(p)$ is available.
Define the \textit{social-gradient flow} as
\begin{equation}
\label{eq:gradient_flow}
\dot p(t) = g_*(p(t))= \nabla \Phi(x)\big|_{x = x^*(p(t))}, \quad p(0) \in \P.
\end{equation}
The social-gradient flow~\eqref{eq:gradient_flow} 
utilizes the negative definiteness of $\Sym(\D x^*(p))$ to ensure that $\nabla \Phi(x^*(p))$ is a descent direction for $\Phi_*$ at $p$. 
For $c \geq 0$, define the sublevel sets
\begin{equation}
  \label{eq:sublevel_sets}
\P_c = \{p \in \P: \Phi(x^*(p)) - \Phi(x^\dagger) \leq c \}.
\end{equation}
Let
\(
c^* = \min_{x \in \partial \X} \Phi(x) - \Phi(x^\dagger) > 0,
\) 
where $\partial \X$ is the boundary of $\X$.
We prove that each $\P_c$, with $c<c^*$, is forward invariant under~\eqref{eq:gradient_flow},
and constitutes a domain of attraction for the optimal incentive $p^\dagger$.

\begin{theorem}
\label{thm:gradient_flow_convergence}
Suppose Assumptions \ref{ass:strong_monotonicity} and \ref{ass:social_cost} hold.
For any $c < c^*$,
the set $\P_c$ is compact and forward invariant under the dynamics~\eqref{eq:gradient_flow}.
Moreover, $\P_c$ contains a unique asymptotically stable equilibrium at $p^\dagger $.
In particular, for any $p(0)\in \P_c$, the trajectory of dynamics~\eqref{eq:gradient_flow} satisfies $\lim_{t\to \infty}p(t) =p^\dagger$. 
\end{theorem}
\begin{proof}
  First, for any $c < c^*$, $\P_c$ is compact.
  Observe that $\P_c\subset \bar \P := -G_0(\X)$ is bounded. Consider a sequence $(p_k)_k \subset \P_c$ converging to $p_\infty \in \bar \P $.
  If $p_\infty$ is on the boundary of $\bar \P$, i.e. $p_\infty \in \partial \bar \P $, 
  then $\lim_{k \to \infty} \left(\Phi(x^*(p_k)) - \Phi(x^\dagger) \right) \geq c^*.$
  However, since $\Phi(x^*(p_k)) - \Phi(x^\dagger) \leq c < c^*$ for all $k$, we have a contradiction.
  Therefore $p_\infty \in \P$, and by the continuity of $\Phi\circ x^*$ on $\P$, it holds that 
  \(\Phi(x^*(p_\infty))  = \lim_{k \to \infty} \Phi(x^*(p_k))\leq c + \Phi(x^\dagger) ,\)
  which implies $p_\infty \in \P_c$.
  
 Second, $\P_c$ is forward invariant and contains a unique asymptotically stable equilibrium at $p^\dagger$.
  Define the Lyapunov function $V(p) = \Phi_*(p) - \Phi(x^\dagger)$
  with the sublevel set \(\P_c = \{p\in \P: V(p) \leq c\}.\)
  By Assumption \ref{ass:social_cost}, $V(p) > 0 $ for all $p \in \P_c \setminus\{p^\dagger\}$.
  Moreover, $V$ is non-increasing along trajectories of \eqref{eq:gradient_flow}, since 
  \begin{equation}
  \label{eq:thm1_invariance}
    \dot{V}(p)= \nabla \Phi(x^*(p))^\top \!\Sym (\D x^*(p)) \nabla \Phi(x^*(p)) \overset{(a)}{\leq} 0,
  \end{equation}
  where $(a)$ follows from Lemma~\ref{lemma:NE_properties}(ii). Equality in~\eqref{eq:thm1_invariance} holds only for $p=p^\dagger$, and the result follows.
\end{proof}

 The implementation of~\eqref{eq:gradient_flow} requires exact observations of $x^*(p)$, which assumes that agents instantly compute and play the INE. In practice,
 the planner instead observes a trajectory of responses
 that asymptotically converges to \(x^*(p_k)\) on a faster timescale. In this case, law~\eqref{eq:gradient_flow} emerges as the limiting dynamics of an associated two-timescale iteration.

\section{Incentive Adaptation under Equilibrium Tracking Dynamics}
\label{sec:two_timescale}

This section generalizes the equilibrium observation model and considers agents that adapt according to some unknown learning dynamics. In response to an incentive $p_k$ issued by the planner, the agents follow an arbitrary equilibrium-seeking process $x_{k+1} = (1-a_k)x_k + a_kf(x_k,p_k)$, which asymptotically converges to the equilibrium induced by the incentive. Simultaneously, the planner observes $x_k$ and updates the incentive on a slower timescale using the social-gradient flow.
The overall update dynamics are
\begin{subequations}
\label{eq:iteration}
\begin{align}
  x_{k+1} & = x_k + a_k \left( f(x_k, p_k) - x_k \right), \label{eq:iterations_x}\\
  p_{k+1} & = p_k + \beta_k \nabla \Phi(x_k) \mathbf 1 \big\{p_k + \beta_k \nabla \Phi(x_k) \in  \P_{ c}\big\}, \label{eq:iterations_p}
  \end{align}
\end{subequations}
where $a_k \in (0,1]$ and $\beta_k > 0$. 
In~\eqref{eq:iteration}, $f(\cdot, p_k)$ represents the agents' (unknown) dynamics and $\P_{ c}$ is a known compact sublevel set of $\P$ with $0 < c < c^*$. 
Since $\P_c$ is nonconvex, a projection onto it is intractable and the indicator update in~\eqref{eq:iterations_p} is preferable.

\begin{remark*}
Knowledge of a sublevel $\P_c$ is a necessary compromise to extend local convergence results to all initial conditions in $\X \times \P_c$. This 
requirement is partially fulfilled when the planner has a conservative estimate of an incentive region containing $p^\dagger$.
\end{remark*}

\begin{assumption}
\label{ass:fast_dynamics}
There exists a positive scalar $L_f>0$ such that $f:\X\times \P_{ c} \to \X$ is jointly $L_f$-Lipschitz.
Moreover, there exists a class-$\mathcal{KL}$ function $\phi$,
such that every solution $x(t)$ of dynamics $\dot x = f(x,p)-x$ satisfies
\[\|x(t) - x^*(p)\|_2 \leq \phi \left( \|x(0)-x^*(p)\|_2, t\right), \]
for any $t \geq 0$, $x(0) \in \X$, and $p \in \P_c$.
\end{assumption}

Class-$\mathcal{KL}$ functions are defined in~\cite[Definition 4.3]{khalilNonlinearSystemsHauptbd2002}.
Assumption~\ref{ass:fast_dynamics} holds for a broad class of algorithms
such as the Nash equilibrium (NE) update~\cite{ratliffAdaptiveIncentiveDesign2021},
the projected gradient (PG) update~\cite{tatarenkoGeometricConvergenceGradient2021}, and
the best-response (BR) update~\cite{facchineiFiniteDimensionalVariationalInequalities2004}.

\begin{assumption}
  \label{ass:timescale}
  The sequences $(a_k)_k $ and $(\beta_k)_k$ satisfy 
  \[\sum_{k=0}^\infty a_k = \sum_{k=0}^\infty\beta_k = \infty, \, \sum_{k=0}^\infty(a_k^2 +\beta_k^2 )< \infty, \, \text{and} \, \beta_k = o(a_k ).\]
\end{assumption}

By Lemma~\ref{lemma:NE_existence} and Theorem~\ref{thm:gradient_flow_convergence}, the pair $(x^\dagger, p^\dagger)$ is the unique fixed point of~\eqref{eq:iteration} in $\X \times \P_c$.
Theorem~\ref{thm:ttsa_convergence}
establishes that~\eqref{eq:iteration}
converges to the socially optimal pair $(x^\dagger, p^\dagger)$ from
any initial condition in $\X \times \P_c$. 

\begin{theorem}
\label{thm:ttsa_convergence}
Suppose Assumptions \ref{ass:strong_monotonicity}-\ref{ass:timescale} hold.
Then, for any initial condition $(x_0, p_0) \in \X \times \P_{ c}$, the sequence $(x_k, p_k)_{k\geq 0}$ generated by iteration \eqref{eq:iteration} 
satisfies
\begin{equation}
  \begin{aligned}
    \lim_{k \to \infty}\|x_k - x^*(p_k)\|_2 &= 0, \text{ and }
    \lim_{k \to \infty}p_k = p^\dagger.
\end{aligned}
\end{equation}
Consequently, $(x_k, p_k) \to (x^\dagger, p^\dagger)$, as $k \to \infty$.
\end{theorem}

\begin{remark*}
    Under Assumptions~\ref{ass:fast_dynamics} and~\ref{ass:timescale}, iteration~\eqref{eq:iteration}
    can be viewed as a two-timescale stochastic approximation  (TTSA) scheme~\cite{borkarStochasticApproximation2008}. 
    However, its convergence cannot be directly established by standard arguments
    due to the non-convexity of $\P_c$ and the use of the indicator in~\eqref{eq:iterations_p}.
    \end{remark*}

We preface the proof of Theorem~\ref{thm:ttsa_convergence} with three lemmas. First, 
Lemma~\ref{lemma:fast_tracking} proves that~\eqref{eq:iterations_x} tracks the slow-moving equilibria $(x^*(p_k))_k$ on a faster timescale, satisfying $\|x_k- x^*(p_k)\|_2 \to 0$, even when $\P_c$ is nonconvex. 
Second, Lemma~\ref{lemma:PE_trial_2} establishes that the indicator in~\eqref{eq:iterations_p} is activated infinitely often for any initial condition in $\X \times \P_c$, so it is asymptotically negligible.
Finally, Lemma~\ref{lemma:nonconvex_sa} establishes that standard stochastic approximation for asymptotic pseudo-trajectories extend to schemes over compact forward-invariant domains, without requiring convexity.

\begin{lemma}
\label{lemma:fast_tracking}
Suppose Assumptions~\ref{ass:strong_monotonicity}--\ref{ass:timescale} hold.
Let $(x_k, p_k)_{k \geq 0}$ be generated by dynamics~\eqref{eq:iteration}
with $(x_0, p_0) \in \mathcal{X} \times \mathcal{P}_c$.  Then
\(
  \|x_k - x^*(p_k)\|_2 \to 0,
\)
as $k \to \infty$.
\end{lemma}
\begin{proof}
Let $\delta_k = \beta_k \mathbf{1}\{p_k+\beta_k \nabla \Phi(x_k) \in \P_c \}$, and 
$h(x,p) = f(x,p) - x$ so~\eqref{eq:iteration}
can be written as
\begin{subequations}
\label{eq:iteration_augmented_time1}  
\begin{align}
    x_{k+1} & = x_k + a_k h(x_k,p_k) \label{eq:iteration_augmented_time1_x}, \\
    p_{k+1} & = p_k + \delta_k \nabla \Phi(x_k). \label{eq:iteration_augmented_time1_p} 
\end{align}
\end{subequations}

If $\P_c$ is convex,~\eqref{eq:iteration_augmented_time1} is a TTSA scheme and the result follows from established techniques~\cite[p. 66, Lemma~1]{borkarStochasticApproximation2008}.
We show that the result holds if $\P_c$ is nonconvex.
Let
\(M = \sup_{x \in\mathcal{X}}\|\nabla\Phi(x)\|_2\), 
\(B = \sup_{\mathcal{X}\times\mathcal{P}_c}\|h(x,p)\|_2,\)
and 
\(
  m(n,T) = \max\{k \geq n : \sum_{j=n}^{k-1} a_j \leq T\}.
\)
The proof proceeds in three steps.

\emph{Step 1:}
Since $\delta_k \leq \beta_k = o(a_k)$, 
for every $\varepsilon > 0$ 
there exists some $N_\varepsilon$ such that $\delta_k \leq \varepsilon a_k$,
$k \geq N_\varepsilon$.  For $n \geq N_\varepsilon$ and $k \in [n,\, m(n,T)]$,
\begin{equation}
  \label{eq:slow_frozen}
  \|p_k - p_n\|_2
  \;\leq\; M \sum_{j=n}^{k-1}\delta_j
  \;\leq\; \varepsilon M \sum_{j=n}^{k-1} a_j
  \;\leq\; \varepsilon MT.
\end{equation}
By~\eqref{eq:slow_frozen}, it holds that $\sup_{n \leq k \leq m(n,T)}\|p_k - p_n\|_2 \to 0$ as $n \to \infty$ for each $T$ fixed. 

\emph{Step 2:}
Let $s_k = \sum_{j =0}^{k-1} a_j$ and denote with $\bar{x}(s)$ the piecewise-linear
interpolation of $(x_k)_k$ on $[s_k, s_{k+1}]$ and $\X$. 
Let $x^{n}(s)$, $s \geq n$ be the trajectory of
$\dot{x} = h(x,p_n)$ with $x^n(0) = x_n$.  
By arguments similar to~\cite[p. 12, Lemma~1]{borkarStochasticApproximation2008},
\begin{equation}
  \label{eq:fast_APT}
  \sup_{0 \leq s \leq T} \|\bar{x}(s_n + s) - x^n(s)\|_2 \to 0, \text{ as } n\to \infty.
\end{equation}

\emph{Step 3:}
By the $\mathcal{KL}-$bound of Assumption~\ref{ass:fast_dynamics}, property~\eqref{eq:fast_APT} and the Lipschitz continuity of $x^*$, we conclude that for any $\epsilon >0$ and an appropriate selection of $T_\epsilon$,
\[\|x_k-x^*(p_k)\|_2 \leq a_{k-1}B + \epsilon + \frac{2}{m}\sup_{n \leq k \leq m(n,2T_\epsilon)}\|p_k - p_n\|_2.\]
 As $k \to \infty$, the first and third terms vanish.
 Since $\epsilon >0$ is arbitrary, the claim follows.
\end{proof}

\begin{lemma}
\label{lemma:PE_trial_2}
    Suppose Assumptions \ref{ass:strong_monotonicity}-\ref{ass:timescale} hold.
    Let $(x_k,p_k)_{k\geq 0}$ be a sequence generated by iteration~\eqref{eq:iteration} with initial condition $(x_0,p_0) \in \X \times \P_{ c}$.
    It then holds that
    \[\sum_{k=0}^\infty \beta_k \mathbf{1}\{p_k + \beta_k \nabla \Phi(x_k) \in \P_c\} = \infty.\]
    \end{lemma}
\begin{proof}
  At each $k$, denote $\hat p_k = p_k + \beta_k\nabla \Phi(x_k)$.
  Suppose for some initial condition the sequence $\left( x_k,p_k \right)_{k \geq 0}$ satisfies
  \begin{equation}
  \label{eq:lemmaPE_contradiction}
    \sum_{k=0}^{\infty} \beta_k \mathbf{1}\{\hat p_k \in \P_c\} < \infty
  .\end{equation}
  Then~\eqref{eq:iterations_p} implies 
  \(
    \sum_{k=0}^{\infty} \|p_{k+1} \!-\! p_k \|_2 < \!\infty,
  \)
  so $p_k \to p_\infty$, where $p_\infty \in \P_c$ depends on $\left( x_0, p_0\right)$.

  First,~\eqref{eq:lemmaPE_contradiction} implies a contradiction if $p_\infty \in \interior \P_c$.
  Let $r > 0$ be such that the open ball of radius $r$ around $p_\infty$ remains in $\P_c$.
  We have that
  \(
    \|\hat p_k - p_\infty \|_2 \leq \|p_k - p_\infty \|_2 + \beta_k M,
  \)
  where 
  $M = \sup_{x\in\X} \|\nabla \Phi(x)\|_2$,
  so there exists a $T$ such that $\|\hat p_{T}-p_\infty\|_2 < r$, and hence $\mathbf{1}\{\hat p_k \in \P_c\} = 1 $ for all $k \geq T$,
  which contradicts~\eqref{eq:lemmaPE_contradiction}.

  Second,~\eqref{eq:lemmaPE_contradiction} implies a contradiction when $p_\infty \in \partial\P_c$.
  Let $V(p) = \Phi_*(p) -\Phi(x^\dagger)$.
  For large $k$,  $\hat p_k = p_k + \beta_k \nabla \Phi(x_k) \in \P$ since $p_k \in \P_c$,
  so by a Taylor expansion
  \begin{align}
  & V(\hat p_k)  = V(p_k) + O(\beta_k^2) \notag \\
  &  + \beta_k \nabla \Phi(x^*(p_k))^\top \! \D x^*(p_k) \nabla \Phi(x^*(p_k)) \tag{\textrm{i}} \label{eq:lemmaPE_2}\\ 
  &  + \beta_k \nabla \Phi(x^*(p_k))^\top \! \D x^*(p_k) \big(  \nabla \Phi(x_k) \! - \! \nabla \Phi(x^*(p_k))\big). \tag{\textrm{ii}} \label{eq:lemmaPE_3} 
  \end{align}

  For term \eqref{eq:lemmaPE_2} by Assumption~\ref{ass:strong_monotonicity}, for any $w \in \mathbb{R}^n$,
  \[ 
  \begin{aligned}
    w^\top \Sym \left( \D x^*(p_k) \right) w  
    & \overset{(a)}{=} -v^\top \Sym \left( \D G_0(x^*(p_k)) \right) v \\
    & \overset{(b)}{\leq} - \frac{m}{2}\|v\|_2^2  \overset{(c)}{\leq} -\frac{m}{2} \sigma_1^2 \|w\|_2^2,
\end{aligned} 
  \]
  where $v = \D G_0(x^*(p_k))^{-1} w $, $(a)$  follows from
   $Dx^*(p_k) = -DG_0(x^*(p_k))^{-1}$
  and the identity $A^{-\top}+ A^{-1} = A^{-\top}(A+A^\top)A^{-1}$, $(b) $ from Assumption~\ref{ass:strong_monotonicity}, and
  $(c)$ from
  \[\|w\|_2 = \| DG_0(x^*(p_k)) v \|_2 \leq  \frac{1}{\sigma_1}  \|v\|_2.\]
  By Lemma~\ref{lemma:NE_properties} and Assumption~\ref{ass:social_cost}, term~\eqref{eq:lemmaPE_3} satisfies
  \[
    \eqref{eq:lemmaPE_3} \leq \beta_k \|\nabla \Phi(x^*(p_k)) \|_2 \sigma_2  \LDPhi \|x_k - x^*(p_k) \|_2.
  \]
  Conclude that
  \begin{align*}
   V(\hat p_k)  \leq V(p_k) &+ O(\beta_k^2)
  - \beta_k  \frac{m}{2}\sigma_1^2\|\nabla \Phi(x^*(p_k))\|_2^2  \\
  &  + \beta_k \sigma_2 \LDPhi\|\nabla \Phi(x^*(p_k))\|_2  \|x_k-x^*(p_k)\|_2.
  \end{align*}
  By Lemma~\ref{lemma:fast_tracking}, $\|x_k-x^*(p_k)\|_2 \to 0$ and
  $\|\nabla \Phi(x^*(p_k))\|_2 \to \|\nabla \Phi(x^*(p_\infty))\|_2 > 0$, since $p_\infty \in \partial \P_c$.
  For sufficiently large $k$, the negative term dominates, implying
  \( 
   V(\hat p_k) < V(p_k) \leq c.\)
  Hence, $\hat p_k \in \P_c$ for sufficiently large $k$, implying
  \(
  \sum_{k=0}^{\infty} \beta_k \mathbf{1}\{p_k + \beta_k \nabla \Phi(x_k) \in \P_c\} = \infty.
  \)
\end{proof}

\begin{lemma}
\label{lemma:nonconvex_sa}
Let $\P_c$ be a compact (nonconvex) subset of an open $U \subset \R^n$, and let $g:U \to \R^n$ be $L_g$-Lipschitz.
Suppose $\P_c$ is forward invariant for dynamics $\dot{p} = g(p)$ and contains a unique globally asymptotically stable $p^\dagger \in \interior(\P_c).$  Let $(p_k)_{k\geq 0} \subset \P_c$ satisfy $p_{k+1} = p_k + a_k g(p_k)$, where
$a_k \geq 0$,  $\sum_{k=0}^\infty a_k = \infty$, and $\sum_{k=0}^\infty a_k^2 < \infty$. Then $\lim_{k\to \infty}p_k = p^\dagger$.
\end{lemma}
\begin{proof}
    Define $t_k = \sum_{j=0}^{k-1}a_j$ and let $\bar p:[0, \infty) \to \R^n$ be the linear interpolation of $(p_k)_k$ on $t\in[t_k, t_{k+1}]$, that is,
    \begin{equation}
    \label{eq:nonconvex_sa_interpolation}
        \bar p(t) = p_k + \frac{t-t_k}{a_k}(p_{k+1}-p_k), \quad t \in [t_k, t_{k+1}].
    \end{equation}
    Since $\P_c$ is non-convex, $\bar p(t)$ belongs to the convex hull $\operatorname{co}(\P_c)$ and satisfies $\bar p(t_k) \in \P_c$. Then
    \begin{equation}
    \label{eq:nonconvex_sa_temp1}
        \|\bar p(t)- p_k\|_2\leq a_k B, \quad t \in [t_k, t_{k+1}],
    \end{equation}
    where $B = \sup_{p \in \P_c}\|g(p)\|_2$.
    Let $\epsilon > 0$ be such that
    \(\P_c^\epsilon = \{p: \dist(p, \P_c) \leq \epsilon\} \subset U.\) Since $a_k \to 0$, there exists some $K$ so that $a_kB \leq \epsilon$ for all $k \geq K$.
    Then~\eqref{eq:nonconvex_sa_temp1} implies $\bar p(t) \in \P_c^\epsilon$ for all $t\geq t_K$. Hence $g(\bar p(t))$ is well-defined for $t \geq t_K$.

    For each $n\geq K$, let $p^n(t)$ be the trajectory of $\dot p = g(p)$ with $p^n(0)=p_n$. The error
    $e_p(t) = \bar p(t_n+t) -p^n(t)$ satisfies
    \[\|e_p(t)\|_2 \leq L_g \int_0^t \|e_p(\tau)\|_2\, d\tau + L_g \max_{k\geq n}a_k B t.\]
    By a similar argument to~\cite[p. 12, Lemma~1]{borkarStochasticApproximation2008}, it follows that for all $T>0$,
    \[\sup_{t\in [0,T]} \|\bar p(t_n+t) - p^n(t) \|_2 \to 0, \text{ as } n\to \infty.\]
    The standard convergence result for asymptotic pseudotrajectories then implies convergence to a compact, connected, internally chain-transitive invariant set of $\dot p = g(p)$ (see e.g.~\cite[p. 15, Theorem~2]{borkarStochasticApproximation2008}).
    \end{proof}

The proof of Theorem~\ref{thm:ttsa_convergence} is finalized below.
\begin{proof}[Proof of Theorem~\ref{thm:ttsa_convergence}]
    Let $\delta_k = \beta_k \mathbf{1}\{p_k + \beta_k \nabla \Phi(x_k) \in \P_c\}$.
    By Lemma~\ref{lemma:fast_tracking}, it holds that $\|x_k - x^*(p_k)\|_2 \to 0$, so it remains to show $p_k \to p^\dagger$.

    Define $g_*(p) = \nabla\Phi(x^*(p))$ in~\eqref{eq:gradient_flow}, so~\eqref{eq:iterations_p} becomes
    \[
    \begin{aligned}
        p_{k+1} & = p_k + \delta_k\bigl(g_*(p_k) + \xi_k\bigr), \\
        \xi_k & = \nabla\Phi(x_k) - \nabla\Phi(x^*(p_k)),
    \end{aligned}
    \]
    where $\|\xi_k\|_2=o(1)$ because $\|\xi_k\|_2 \leq \LDPhi\,\|x_k - x^*(p_k)\|_2$, and $g_*$ is $\Lg$-Lipschitz on $\P$, with $\Lg = 2\LDPhi/m$.
 
    Define $t_k = \sum_{j = 0}^{k-1}\delta_j$, so $\lim_{k\to \infty}t_k = \infty$ by Lemma~\ref{lemma:PE_trial_2}.
    Let $\bar{p}(t)$ be the interpolation of $(p_k)_k$, as in~\eqref{eq:nonconvex_sa_interpolation},
    so
    \[\begin{aligned}
        \|\bar p(t) - p_k \|_2 & \leq \delta_k \|g_*(p_k)+ \xi_k\|_2  \leq \delta_k (B+ \|\xi_k\|_2),
    \end{aligned}\]
    where $B = \sup_{p\in \P_c}\|g_*(p)\|_2$ and $\|\xi_k\|_2 = o(1)$ is a vanishing perturbation. Therefore,
    $\bar p(t)$ is an asymptotic pseudotrajectory of the dynamics $\dot{p} = g_*(p)$, 
    and the result follows from Lemma~\ref{lemma:nonconvex_sa}.
\end{proof}

\section{Numerical Examples}
\label{sec:numerical}

This section presents two illustrative games that validate the convergence results in Theorems~\ref{thm:gradient_flow_convergence} and~\ref{thm:ttsa_convergence}. In both examples, the social cost is $\Phi: \X \to \R$, $\Phi(x) = \frac{1}{2}\|x-x^\dagger\|_2^2$, and $x^\dagger \in \interior \X$. 
Appendix~\ref{sec:supplementary} verifies relevant assumptions and provides additional implementation details.

\subsection{Aggregative Game over a Directed Network}

Consider an aggregative game among $n=5$ players in a directed network, which arises in demand-response and distributed energy-resource problems.
Each player has 
cost
\begin{equation}
\label{eq:aggregative_cost}
\ell_i(x) = \frac{q_i}{2} x_i^2 + a\sum_{j=1}^n w_{ij }x_ix_j,
\end{equation}
where $x \in \X = [-2,2]^n$. The marginal utility parameters $q_i > 0$ represent agent-specific preferences, the matrix $W = [w_{ij}]_{i,j=1}^n$ is a stochastic (non-symmetric) adjacency matrix with zero diagonal, and the scalar 
$a > 0$ is the coupling strength between agents.
The nominal pseudo-gradient of this game is $G_0(x) = Mx$, where $M = Q + a W$ and $Q = \Diag(q_i)$. Assumption~\ref{ass:strong_monotonicity} is satisfied when $a< \lambda_{\min}(Q) / \|\Sym(W)\|_2$, as shown in Appendix~\ref{sec:supplementary_aggregative}.

We first verify Theorem~\ref{thm:gradient_flow_convergence} for $100$ initial conditions, sampled uniformly from $\interior \P_{c^*}$.
Figure~\ref{fig:gradient_flow} reports social cost $\Phi(x^*(p(t)))$ and incentive error $\|p(t) - p^\dagger\|_2$ for two initial conditions, corresponding to the initial conditions with the smallest and largest sampled values of $\Phi_*$.
The corresponding sublevel sets are forward invariant, and both 
trajectories converge to $p^\dagger$. 

\begin{figure}
    \centering
    \includegraphics[width=0.9\linewidth]{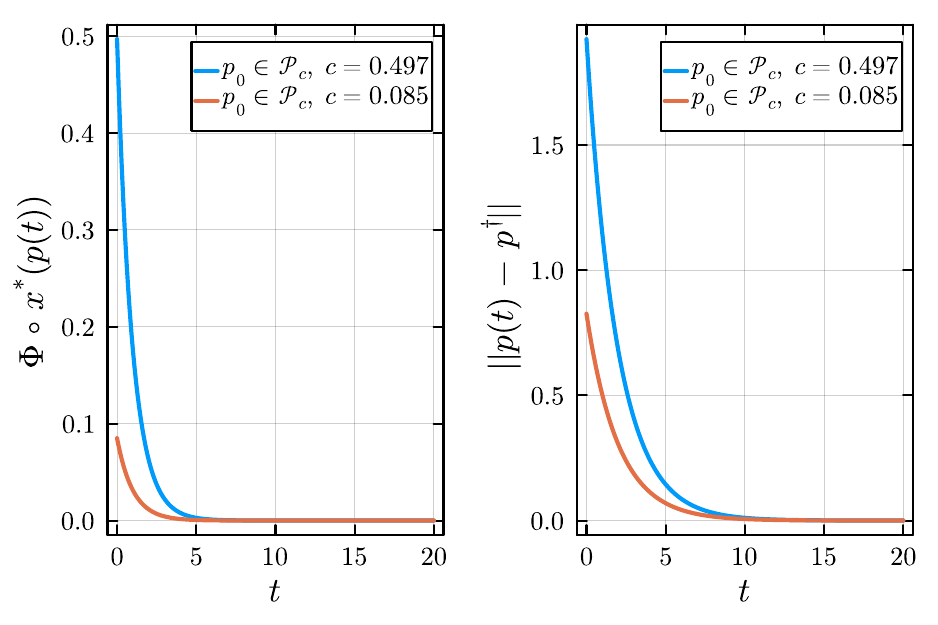}
    \caption{$\Phi(x^*(p(t)))$ and $\|p(t)-p^\dagger\|_2$, under~\eqref{eq:gradient_flow} with initial conditions in $ \interior \P_{c^*}$. Trajectories correspond to flows from the maximal and minimal sublevel sets across 100 initial conditions. }
        \label{fig:gradient_flow}
\end{figure}

Moreover, we verify Theorem~\ref{thm:ttsa_convergence} under different learning dynamics.
We consider agents who adapt according to Nash equilibrium and best-response rules, defined as $f^{\textrm{NE}}(x,p) = x^*(p)$ and $f^{\textrm{BR}}(x,p) = - Q^{-1}(p+aWx)$, respectively. One can verify that Assumption~\ref{ass:fast_dynamics} holds for each rule (see Appendix~\ref{sec:supplementary_aggregative}). 
Figure~\ref{fig:NetAgg_combined_convergence} presents the evolution of the two-timescale error $\|x_k - x^*(p_k)\|_2$ and the incentive error $\|p_k-p^\dagger\|_2$. Results are obtained over 100 initial conditions in $\X \times \P_c$ with $c = 0.8c^*$.

\begin{figure}
    \centering
    \includegraphics[width=\linewidth]{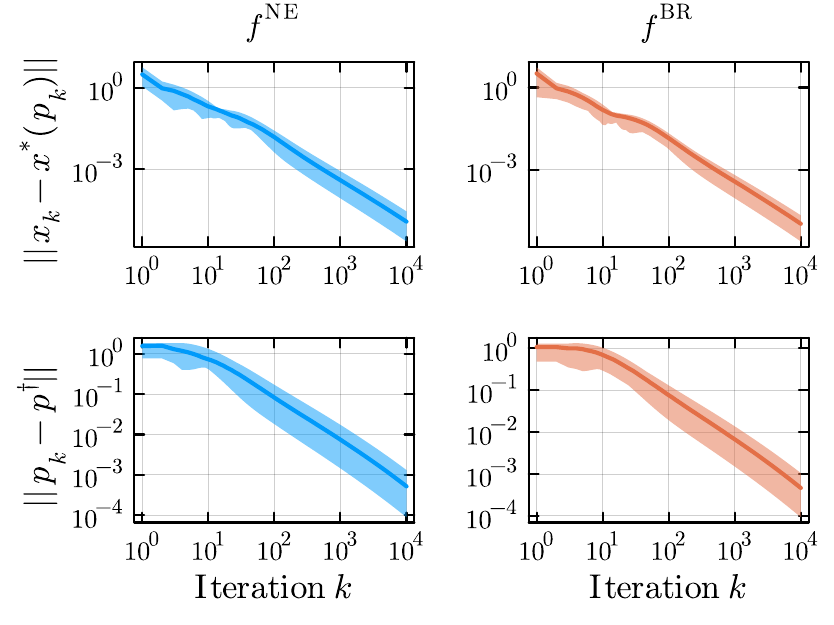}
    \caption{Median (solid) and min–max envelope (shaded) of $\|x_k-x^*(p_k)\|_2$ (top) and $\|p_k-p^\dagger\|_2$ (bottom) under~\eqref{eq:iteration}, with $f^\textrm{NE}$ (left) and $f^\textrm{BR}$ (right), over 100 initial conditions sampled from $\interior \X \times \P_c$, with $c=0.8c^*$.}
    \label{fig:NetAgg_combined_convergence}
\end{figure}

\subsection{Coupled Oscillator Game}

We further investigate a two-player coupled oscillator game adapted from~\cite{ratliffAdaptiveIncentiveDesign2021}. Each player has a private cost
\begin{equation}
\label{eq:oscillator_cost}
\ell_i(x) = -\theta_i \cos(x_i) + \cos (x_1 - x_{2}),
\end{equation}
where $x \in \X = [-\frac{\pi}{3},\frac{\pi}{3}]^2$ and $\theta_i>0$ are scalars.
We select parameters $\theta_1 = 4.2$ and $\theta_2 = 5$, for which Assumption~\ref{ass:strong_monotonicity} holds.
The nominal pseudo-gradient is the nonlinear map
\[G_0(x) = \begin{bmatrix}
    \theta_1 \sin(x_1)  - \sin(x_1-x_2) \\ \theta_2 \sin(x_2) + \sin(x_1-x_2)
\end{bmatrix}.\]
For a given $p$, the response $x^*(p)$ can be solved numerically via an associated convex minimization.
Agents use a projected-gradient (PG) rule,
\(f^{\textrm{PG}}(x,p) = \Pi_{\X}\left( x - \eta G_\gamma(x) \right),\) 
where $\eta > 0$ is a constant and  $\Pi_{\X}$ is the projection onto $\X$. 
Appendix~\ref{sec:supplementary_oscillator} verifies that this game satisfies Assumption~\ref{ass:strong_monotonicity}, derives the convex program used
to compute $x^*(p)$, and establishes that $f^{\textrm{PG}}$ satisfies
Assumption~\ref{ass:fast_dynamics}.

Figure~\ref{fig:Oscillator_trajectories} illustrates the trajectories of 
$(x_k, p_k)_{k \geq 0}$ for
the two-timescale scheme~\eqref{eq:iteration} in the strategy and incentive spaces and the associated social-gradient flow~\eqref{eq:gradient_flow} from $p(0)=p_0$.  Observe that the forward invariance of $\P_{c_0}$, where $c_0= \Phi_*(p_0) - \Phi(x^\dagger)$, is not guaranteed in the two-timescale regime, due to the tracking bias introduced by the difference $x_k-x^*(p_k)$. Nevertheless, incentive iterates remain in $\P_c$, with  $c=0.9c^*$, and converge to $p^\dagger$.
The indicator rejects no updates after the initial transient, as predicted by Lemma~\ref{lemma:PE_trial_2}.

Finally, Figure~\ref{fig:Oscillator_convergence} presents the discretization and incentive errors, $\|x_k-x^*(p_k)\|_2$ and $\|p_k-p^\dagger\|_2$ respectively, for
iteration~\eqref{eq:iteration} while varying the degree of timescale separation $\beta_k / a_k = O(k^{-\gamma})$ for different values of $\gamma>0$. 
Smaller $\gamma$ values introduce a more pronounced tracking bias in the incentive, and larger $\gamma$ values result in diminished finite-time incentive performance.

\begin{figure}
    \centering
    \includegraphics[width=0.9\linewidth]{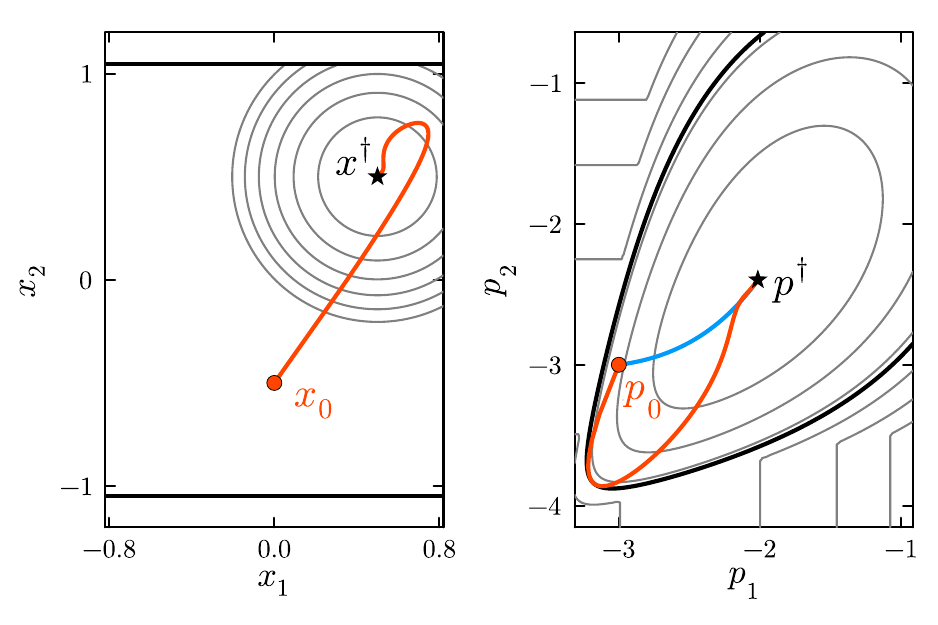}
    \caption{Iterates $(x_k,p_k)_k$ (orange) generated by~\eqref{eq:iteration} from
    $x_0=(0, -0.5)^\top$ and $p_0=(-3, -3)^\top$.
    Blue curve indicates the social-gradient flow~\eqref{eq:gradient_flow}
    from $p(0)=p_0$ and $t\in[0,20]$.
    Contours are level sets of $\Phi(x)$ (left) and $\Phi_*(p)$ (right). 
    Solid lines are the boundaries $\partial \X$ and $\partial \P_{c}$, with $c =0.9c^*$.
    } 
    \label{fig:Oscillator_trajectories}
\end{figure}

\begin{figure}
    \centering
    \includegraphics[width=0.9\linewidth]{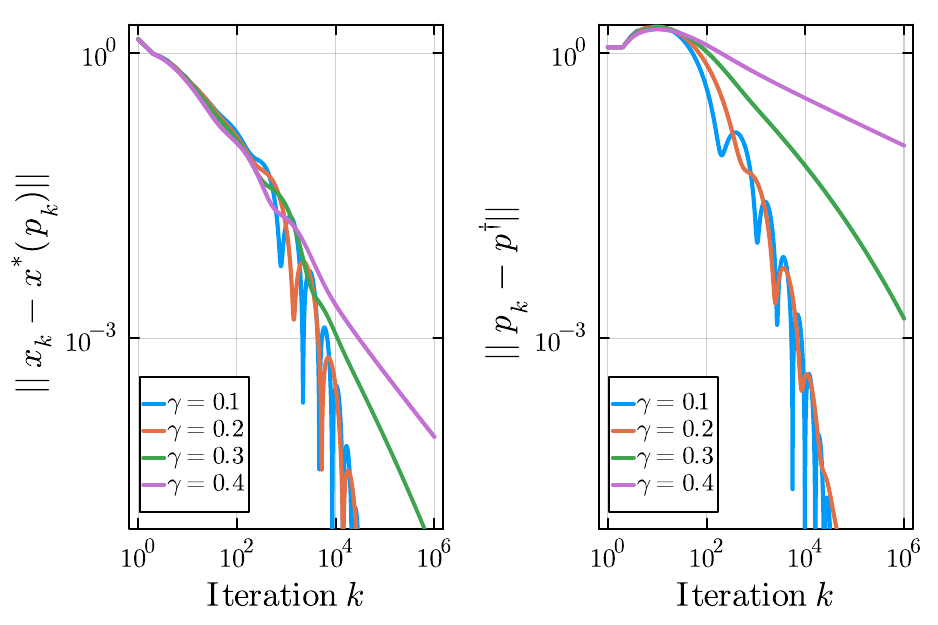}
    \caption{$\|x_k-x^*(p_k)\|_2$ and $\|p_k-p^\dagger\|_2$ under iteration~\eqref{eq:iteration}, with initial condition $x_0=(0, -0.5)^\top$ and $p_0=(-3, -3)^\top$.
    Lines depict variations of the timescale separation degree $\beta_k / a_k = O(k^{-\gamma})$ for different $\gamma>0$.}
    \label{fig:Oscillator_convergence}
\end{figure}

\section{Concluding Remarks}
\label{sec:conclusion}

This work studies adaptive incentive design under information asymmetry, where the planner observes either the Nash equilibrium response or a transient trajectory that asymptotically tracks the induced Nash equilibrium.
Our first contribution is to establish a social-gradient dynamical system in continuous time for seeking the optimal incentive
which does not require the equilibrium sensitivity information and still yields a descent direction for the planner's objective.
Our second contribution 
is a two-timescale implementation  of the proposed social-gradient dynamics,
in which agents respond on a faster timescale than the planner. 
Its convergence to the socially optimal pair is proven for any agent learning rule that converges uniformly.

We consider two possible extensions of the present analysis. First, one may consider a social-gradient incentive rule under a relaxed assumption on the local costs (Assumption~\ref{ass:strong_monotonicity}). Second, the proposed two-timescale method operates with knowledge of some feasible set $\P_c$. One may investigate whether this requirement can be relaxed to knowledge of the set $\X$ alone, for example through an algorithm that incrementally refines the operating domain from an initial conservative estimate of the safe set.

\bibliographystyle{IEEETran} 
\bibliography{bibliography}

\appendices
\section{Proofs of Supporting Lemmas} 
\label{app:proofs}
\subsection{Proof of Lemma \ref{lemma:NE_existence}}
\begin{proof}
We prove the lemma in two steps.

\textit{Step 1.} $x^*(\cdot)$ is uniquely defined on $\mathbb R^n$.
For any $x,y\in \X$, denote $v = x-y$. Since $\X$ is convex, $y + tv \in \X$ for all $t\in [0,1]$. By the fundamental theorem of calculus on $t \mapsto G_0(y+tv)$ and Assumption \ref{ass:strong_monotonicity},
\[v^\top (G_0(x) - G_0(y)) = \int_0^1 v^\top \D G_0(y+tv) v \,dt \geq \frac{m}{2}\|v\|_2^2.\]
Therefore $G_0 (x)$ is $\frac{m}{2}-$strongly monotone on $\X$, and so is $G_\gamma(x)$ for any $p$. Hence~\eqref{eq:NE_condition} has a unique solution $x^*(p) \in \X$ (see~\cite[Theorem 2.3.3]{facchineiFiniteDimensionalVariationalInequalities2004}), 
so the map is uniquely defined. 

\textit{Step 2.} The map $x^*(\cdot)$ restricted to $\mathcal P$ is a diffeomorphism. By Assumption \ref{ass:strong_monotonicity}, $\D G_0(x)$ is nonsingular 
for every $x\in \interior\X$ and $G_0$ is injective by strong monotonicity.
It follows that
$G_0:\interior \X \to G_0(\interior \X)$ is a diffeomorphism from $\interior \X$ to $-\P =G_0(\interior \X)$.
Since $G_0$ is an open map, $\P$ is open, and
$x^*(p) = G_0^{-1}(-p)$ defines a $\mathcal{C}^1$ diffeomorphism 
$x^*:\P\to\interior\X$.
\end{proof}

\subsection{Proof of Lemma \ref{lemma:NE_properties}}
\begin{proof}
  \textit{Claim (i).}
  Since $x^* \in \interior \X$, 
  $x = x^* + t d \in \X$ for all small $t>0$ and $d \in \R^n$. By~\eqref{eq:NE_condition}, we have $G_\gamma(x^*)^\top d \geq 0$ for all $d \in \R^n$, which implies $G_\gamma(x^*) = 0$. Further, by 
  Assumption~\ref{ass:strong_monotonicity}, $G_0$ is $\frac{m}{2}$-strongly monotone on $\X$, therefore, for any $p_1, p_2 \in \P$, we have 
  \[\begin{aligned}
     & \frac{m}{2} \|x^*(p_1)-x^*(p_2)\|_2^2 \\ & \leq \left( G_0(x^*(p_1)) - G_0(x^*(p_2))\right)^\top \left(x^*(p_1) - x^*(p_2)\right) \\
     & \overset{(a)}{\leq} \|p_1 - p_2 \|_2 \, \|x^*(p_1) - x^*(p_2) \|_2,
  \end{aligned}\]
  where $(a)$ uses $G_\gamma(x^*(p_1)) = G_\gamma(x^*(p_2)) = 0$.
  By Lemma \ref{lemma:NE_existence} and Claim (i), we have that for all $p \in \P$, $x^*= G_0^{-1}(-p)$ and $\D x^*(p)  = -(\D G_0(x^*))^{-1}$.

  \textit{Claims (ii) and (iii).}
  For any $w \in \R^n\setminus \{0\}$, let $v=\D G_0(x^*)^{-1}w$. Using $\D x^*(p)  = -(\D G_0(x^*))^{-1}$ above and the identity $A^{-\top} + A^{-1} = A^{-\top}(A+ A^\top)A^{-1}$ one concludes
  \[\begin{aligned}
    & w^\top \Sym(\D x^*) w = - v^\top \Sym(\D G_0(x^*)) v < 0, \text{ and} \\
    &\frac{m}{2} \|w\|_2^2  \leq w^\top\Sym (\D G_0(x^*)) w \leq   \|w\|_2 \|\D G_0(x^*)w\|_2.
  \end{aligned} \]
  It follows that $\sigma_{\max}(\D x^*) = \sigma_{\min}(\D G_0(x^*))^{-1} \leq \frac{2}{m }$,
  and 
  \(\sigma_{\min}(\D x^*) \geq (\max_{x\in X} \sigma_{\max}(\D G_0(x)))^{-1}.\)

  \textit{Claim (iv).}
  Given $p_1$, $p_2 \in \P$, let $A_i = \D G_0(x^*(p_i))$.
  Applying $\D x^*(p_i)  = -(\D G_0(x^*(p_i)))^{-1} = -A_i^{-1}$ and the matrix property $A^{-1}-B^{-1} = A^{-1}(B-A)B^{-1}$,
  we have that 
  \(
    \begin{aligned}
      \D x^*(p_1) - \D x^*(p_2) = A_2^{-1} \left( A_1 - A_2 \right) A_1^{-1},
    \end{aligned}
  \)
  so together with Claims (i) and (iii) we conclude
  \[
  \begin{aligned}
    & \|\D x^*(p_1) - \D x^*(p_2)\|_2 \leq \|A_2^{-1}\|_2 \|A_1 - A_2\|_2 \|A_1^{-1}\|_2 \\
    & \leq \frac{4\LDG}{m^2}  \|x^*(p_1) - x^*(p_2)\|_2 
    \leq \frac{8\LDG}{m^3} \|p_1 - p_2\|_2.
  \end{aligned}
  \]
\end{proof}

\section{Supplement on Numerical Examples}
\label{sec:supplementary}

This appendix provides the supporting details for the numerical examples in Section~\ref{sec:numerical}. For each example, we verify the strong-monotonicity condition in Assumption~\ref{ass:strong_monotonicity}, characterize the computation of the equilibrium response map $x^*(\cdot)$, and show that the agent learning dynamics satisfy Assumption~\ref{ass:fast_dynamics}. 

\subsection{Aggregative Game over a Directed Network}
\label{sec:supplementary_aggregative}
We first give a sufficient condition such that Assumption~\ref{ass:strong_monotonicity} is satisfied.
Observe that 
\begin{equation}
\label{eq:example1_pd}
\begin{aligned}
\Sym(M) & = Q + a\Sym(W) \\
& \succeq \Big(\lambda_{\min}(Q) - a\|\Sym(W)\|_2\Big) \I\succ 0,
\end{aligned}
\end{equation}
so, Assumption~\ref{ass:strong_monotonicity} holds when $a< \lambda_{\min}(Q) / \|\Sym(W)\|_2$. It follows that $M$ is nonsingular, and the linearity of $G_0$ yields a closed-form response map $x^*(p) = -M^{-1}p$ for all $p \in \P$. 

Moreover, we verify Assumption~\ref{ass:fast_dynamics} holds for each of the learning dynamics.
In the NE-update rule, agents utilize $f^{\textrm{NE}}(x,p) = x^*(p)$ in~\eqref{eq:iterations_x}.
Assumption~\ref{ass:fast_dynamics} is satisfied since dynamics $\dot{x} = x^*(p) - x$ are linear and admit the solution
\(x(t) = x^*(p)  + (x(0)-x^*(p))e^{-t}.\)
In the BR-update rule, each agent minimizes its incentivized cost against the joint strategy of other agents, which yields
\[f_i^{\textrm{BR}}(x,p) = \argmin_{y_i \in \X_i} c^\gamma(y_i, x_{-i}).\]
Because $c_i^\gamma(x)$ are strictly convex with respect to $x_i$, each minimizer $x_i\in \interior \X_i$ satisfies the first-order condition $q_ix_i + a \sum_{j=1}^n w_{ij}x_j + p_i = 0$, which can be expressed
\[f^{\textrm{BR}}(x,p) = - Q^{-1}(p+aWx).\]
To verify Assumption~\ref{ass:fast_dynamics}, consider the error $e = x - x^*(p)$. When $\dot{x} = f^{\textrm{BR}}(x,p) - x$, it follows that $\dot{e} = -Q^{-1}Me$.
Taking the Lyapunov function $V(e) = \frac{1}{2}e^\top Q e$, with the fact that $Q\succ 0$ and $\Sym(M) \succeq \frac{m}{2}\I$ yields
\[
  \dot{V}(e) = -e^\top M e \leq - \frac{m}{2} \|e\|_2^2 \leq - \frac{m}{\lambda_{\max}(Q)} V(e).
\]
Consequently, $V(e(t)) \leq \exp(- \frac{mt}{\lambda_{\max}(Q)})V(e(0))$, and
\[
\|x(t)-x^*(p)\|_2 \leq \sqrt{\frac{\lambda_{\max}(Q)}{\lambda_{\min}(Q)}} \exp(- \frac{m}{2\lambda_{\max}(Q)}t)\|e(0)\|_2.
\]
The bound is independent of $p$, and Assumption~\ref{ass:fast_dynamics} holds.

\subsection{Coupled Oscillator Game}
\label{sec:supplementary_oscillator}
We first present a sufficient condition for Assumption~\ref{ass:strong_monotonicity} to hold. The Jacobian of $G_0$ can be computed
\[ \D G_0(x) =
\begin{bmatrix}
    \theta_1 \cos(x_1)\! - \! \cos(\Delta x) & \cos(\Delta x) \\ \cos(\Delta x) & \theta_2 \cos(x_2) \!- \!\cos(\Delta x)
\end{bmatrix},
\]
where $\Delta x = x_1-x_2$. By Gershgorin's disk theorem,
\[\lambda_{\min}(\D G_0(x)) \geq \min_i\! \big\{\theta_i \cos(x_i) - \cos(\Delta x) - |\cos(\Delta x)|\big\}.\]
Since $|\cos(\Delta x)|\leq 1$ and $\cos(x_i)\geq\cos(\pi/3)$
for all $x\in\X$, Assumption~\ref{ass:strong_monotonicity} holds 
whenever $\theta_i > 2/\cos(\pi/3) = 4$ for $i=1,2$. The selected parameters $\theta_1=4.2$, $\theta_2 = 5$ satisfy this condition.

Regarding the numerical computation of the map $x^*$, observe that $\D G_0(x)$ is symmetric and hence, for every $p$, there exists a $\Psi_p:\X \to \R$ such that $\nabla \Psi_p (x) = G_0(x) + p$ over $\X$ (see~\cite[p.14, Theorem 1.3.1]{facchineiFiniteDimensionalVariationalInequalities2004}). One can  verify that
\[\Psi_p(x) = -\theta_1 \cos(x_1) - \theta_2 \cos(x_2) + \cos(\Delta x) + p^\top x.\]
Therefore, computing the solution $x^*(p)$ to~\eqref{eq:NE_condition} is equivalent to solving the convex program
\(x^*(p) \in\argmin_{x\in \X} \Psi_p(x)\). 

Finally, we verify the $\mathcal{KL}$ bound of Assumption~\ref{ass:fast_dynamics} for the projected-gradient rule $f^{\textrm{PG}}$, in which
\[f_i^{\textrm{PG}}(x,p) = \Pi_{\X_i} \big( x_i - \eta \frac{\partial c^\gamma_i}{\partial x_i}(x)\big),\]
where $\eta >0$ is a constant and $\Pi_{\X_i}$ is the projection on $\X_i$.
Observe that
the fixed point condition 
$x^*(p) = \Pi_{\X}\big(x^*(p) - \eta G_\gamma(x^*(p))\big)$ 
holds. By the non-expansiveness of $\Pi_{\X}$ and strong monotonicity of $G_\gamma$,
\[
\begin{aligned}
\|f^{\textrm{PG}}(x,p)-x^*(p) \|_2^2 \leq (1-\eta m + \eta^2 L^2) \|x-x^*(p)\|_2^2,
\end{aligned}
\]
where $L = \max_{x\in \X} \|\D G_0(x)\|_2$.
Let $\rho = (1-\eta m + \eta^2 L^2)^{\frac{1}{2}}$. When $\eta \in (0,\frac{m}{L^2})$, it holds that $\rho \in (0,1)$.
Applying Grönwall's inequality on $\frac{d}{dt}\frac{1}{2}\|e\|_2^2 \leq -(1-\rho)\|e\|_2^2$
yields 
the required $\mathcal{KL}$ bound, which is uniform in $\P_c$.

\end{document}